\documentclass[12pt, a4paper]{amsart} 

\usepackage{mathptmx} 
\usepackage{amsfonts,amsthm,amssymb,dsfont,stmaryrd,mathrsfs} 
\usepackage[leqno]{amsmath} 
\usepackage{array}
\usepackage{colortbl}
\usepackage{comment} 

\usepackage{graphicx,caption} 
\usepackage[pdfstartview=FitH]{hyperref} 
\usepackage{pdfpages} 
\usepackage{longtable,supertabular, multirow} 
\usepackage{diagbox} 
\usepackage{doi}
\usepackage{tikz, tikz-cd}
\usetikzlibrary{matrix,arrows,decorations.pathmorphing}

\usepackage[hmarginratio={1:1}, vmarginratio={1:1}, top=3cm, bottom=2cm, left=3cm, right=3cm]{geometry} 

\usepackage{setspace}
\usepackage{mathtools} 

\theoremstyle{plain}
\newtheorem*{thm*}{Theorem} 
\newtheorem{thm}{Theorem} [section]
\newtheorem*{thm2*}{Theorem \ref{thmDecNilp}}
\newtheorem*{thm3*}{Theorem \ref{thmDecNilp2}}

\newtheorem{thmA}{Theorem}
\newtheorem{lem}[thm]{Lemma}
\newtheorem*{lem*}{Lemma}
\newtheorem{lemA}{Lemma}
\newtheorem{prop}[thm]{Proposition}
\newtheorem*{prop*}{Proposition}

\newtheorem{cor}[thm]{Corollary}
\newtheorem*{cor*}{Corollary}

\theoremstyle{definition}

\newtheorem*{defn*}{Definition}
\newtheorem{conjecture}{Conjecture}

\newtheorem*{conjecture*}{Conjecture}
\newtheorem{exmp}[thm]{Example}
\newtheorem*{exmp*}{Example}

\newtheorem{prob}{Problem}
\newtheorem*{prob*}{Problem}
\newtheorem{ques}{Question}
\newtheorem*{ques*}{Question}
\newtheorem{remk}[thm]{Remark}
\newtheorem*{remk*}{Remark}

\def\nn{\mathbb{N}}
\def\zz{\mathbb{Z}}
\def\qq{\mathbb{Q}}
\def\rr{\mathbb{R}}

\def\ff{\mathbb{F}}

\def\gcd{{\rm gcd}}

\def\disp{\displaystyle}

\def\tr{{\rm tr}}

\def\mb#1{\mathbf{#1}}
\def\Spec{\textup{Spec}}

\begin{document}

\title{Cyclotomic Matrices and Power Difference Sets}



\author{Wei-Liang Sun}
\address{Department of Mathematics, National Kaohsiung Normal University, Kaohsiung 82444, Taiwan, ROC}
\email{wlsun@mail.nknu.edu.tw}

\maketitle

\begin{abstract}
The cyclotomic matrix is commonly used to arrange cyclotomic numbers in a convenient format. A natural question is whether the structure of the matrix can reflect properties of these numbers. In this article, we examine cyclotomic numbers through their associated cyclotomic matrix and reveal an algebraic structure by relating it to a basis element of a Schur ring. This viewpoint leads to structural identities and reinterpretations of classical results. As an application, we investigate the power difference set problem and establish conditions expressed through cyclotomic matrices, including spectral and determinant characterizations. \\

\end{abstract}

\section{Introduction}
\label{sec1}

Cyclotomic numbers are defined in finite fields as the cardinalities of intersections of certain subsets. Although the concept itself was developed later, its origins trace back to Gauss's study of cyclotomy and the constructibility of regular polygons \cite[Section~358]{Gau}. L.E. Dickson subsequently investigated these numbers in relation to Waring's problem \cite{Dic, Dic3, Dic4}. Building upon these pioneering works, extensive research on cyclotomic numbers has flourished, including applications in combinatorial design, cryptography, and information theory. Additional insights can be found in various surveys \cite{Raj, Kat2, AhmTan}. A conjecture concerning cyclotomic numbers and their connection to {\it conflict-avoiding codes} is discussed in \cite{HsiLiSun, HsiLiSun2}.

Specifically, consider a finite field $\ff_q$ and a divisor $\ell$ of $q-1$. Let $K$ denote the multiplicative subgroup consisting of the $\ell$-th powers of the nonzero elements of $\ff_q$. The $\ell^2$ {\it cyclotomic numbers} are defined as $$(i, j) = \left| (1 + g^i K) \cap g^j K \right|$$ where $0 \leq i, j \leq \ell-1$ and $g$ is a primitive root of $\ff_q$. One of the central challenges in the study of cyclotomic numbers is to determine these values in terms of solutions to a certain Diophantine system, as well as through the discovery of Gauss sums and Jacobi sums, which lead to explicit formulas. However, finding such formulas becomes increasingly difficult as $\ell$ grows, and there appears to be no general method that works uniformly for various values of $\ell$. 

The {\it cyclotomic matrix} is defined as $\left[ (i, j) \right]_{0 \leq i, j \leq \ell-1}$. This matrix is usually presented solely to display cyclotomic numbers for convenience, and to the best of our knowledge, no prior studies have focused on its matrix structure. The motivation of this article is to study cyclotomic numbers through the lens of cyclotomic matrices. To this end, we establish a connection to {\it Schur rings}, which are certain subrings within a group ring. Cyclotomic matrices are then related to matrix representations of basis elements in the Schur ring. This provides an alternative perspective for understanding cyclotomic numbers beyond classical approaches. We prove that a certain type of conjugations of the cyclotomic matrix exhibits behavior similar to that of a basis element in the corresponding Schur ring. The indices $i$ and $j$ of the cyclotomic number $(i, j)$ can naturally be extended to arbitrary integers. Let $k = \frac{q-1}{\ell}$. To ensure consistency and convenience in computation, we index the entries of each matrix in this article from $0$ to $\ell-1$. 

\begin{lemA}
\label{lemA:1}
Denote $A_x = \left[ (i-x, j-x) \right]_{0 \leq i, j \leq \ell-1}$ for $x \in \zz$. For $0 \leq u, v \leq \ell-1$, one has 
\begin{align*} A_u A_v = k \left(\delta_{u, v+\frac{q-1}{2}} I_{\ell} - E_{u+\frac{q-1}{2}, v}\right) + \sum_{w=0}^{\ell-1} (u-v, w-v) A_w, \end{align*} where $I_{\ell}$ is the $\ell$-by-$\ell$ identity matrix, $\delta_{s,t}$ is defined to be $k$ if $s \equiv t \pmod{\ell}$ and $0$ otherwise, and $E_{s,t}$ is the all-zero matrix except for a $1$ in the $(s, t)$-entry.
\end{lemA}

The matrix $A_0$ is the cyclotomic matrix defined previously, and it will also be simply denoted by $A$. Lemma~\ref{lemA:1} will show that every cyclotomic matrix is nearly a {\it normal matrix} (see the discussion preceding Example~\ref{exmp: normality for 7^3}). Moreover, it also leads to the identity involving the sum of squares of cyclotomic numbers, which has also been derived via a different approach in \cite[Lemma~2.2]{BetHirKomMun}. The entries on both sides of the equation in Lemma~\ref{lemA:1} yield an identity among cyclotomic numbers (Theorem~\ref{thm: product of cyclotomic numbers}), which generalizes \cite[Theorem~6.1 and 6.2]{Sna} by Ernst Snapper.

One of the purposes for introducing cyclotomic matrices is to study the problem of {\it power difference sets}, which asks whether the subgroup $K$ consisting of $\ell$-th powers is a {\it difference set} of $\ff_q$--- that is, whether there exists a positive integer $\lambda$ such that every nonzero element of $\ff_q$ can be expressed as a difference of two elements in $K$ exactly $\lambda$ times. (Here we exclude the trivial case $\ell=1$, as $K = \ff_q^{\times}$ in that case.) This problem can be traced back to Raymond Paley \cite{Paley} and Sarvadaman Chowla \cite{Chowla}, and it still remains open. In~\cite{Leh3}, Emma Lehmer rephrased the power difference set problem in term of cyclotomic numbers, and proved that $\ell$ must be even if $K$ is a difference set. Since then, many studies on cyclotomic numbers have followed, and a conjecture has emerged. 

\begin{conjecture*}
If the multiplicative subgroup of $\ell$-th powers forms a difference set of $\ff_p$, then $\ell = 2, 4, 8$.
\end{conjecture*}

Most existing approaches to this conjecture rely on explicit formulas for cyclotomic numbers, and---as previously mentioned---the complexity increases significantly as $\ell$ grows. In order to explore the structural implications behind this conjecture, we investigate the behavior of the cyclotomic matrix when a power difference set arises. Based on Lemma~\ref{lemA:1}, we derive necessary and sufficient conditions for the existence of power difference sets, each of which is independent of the choices of primitive roots.

\begin{thmA}
\label{thmA:1}
Let $K$ be the set of all $\ell$-th powers in $\ff_q^{\times}$, where $\ell \mid q-1$. Let $(i,j)$ denote the cyclotomic numbers with respect to $\ell$ (and a chosen primitive root), and let $A$ be the corresponding cyclotomic matrix. Set $k = |K|$ and $q' = \frac{q-1}{2}$. The following statements are equivalent: 
\begin{enumerate}
\item
$K$ is a difference set of $\ff_q$.
\item
$k$ is odd and $\disp\sum_{i=0}^{\ell-1} (i, j)^2 = \sum_{i=0}^{\ell-1} (i, q')^2$ for some $j$ with $\gcd(j, \ell) = 1$.
\item
$k$ is odd and the main diagonal entries of $A^T A$ are of the form $(a, b, \ldots, b)$.
\end{enumerate}

In this case, $$A^T A = \lambda k J_{\ell} + (k-\lambda) I_{\ell} - k E_{0,0}$$ where $\lambda = \frac{k-1}{\ell}$, $J_{n}$ denotes the $n \times n$ all-one matrix, and $E_{s,t}$ denotes the matrix with a $1$ in the $(s, t)$-entry and zeros elsewhere.

Moreover, if $B$ is the submatrix obtained by removing the entries $\{(i, 0), (q', j)\}_{0 \leq i, j \leq \ell-1}$ from $A$, then $$B^T B = (k - \lambda) (\lambda J_{\ell-1} + I_{\ell-1}).$$ 
\end{thmA} 

The matrix $B$ from the preceding theorem is then utilized to derive the following result, by applying a classical result on regular simplices due to Isaac Jacob Schoenberg \cite[Theorem~1]{Schoenberg}.

\begin{thmA}
\label{thmA:2}
If $K$ is a difference set with $q - k$ a perfect square, then either $\ell$ is divisible by $4$, or $\ell$ is a sum of two odd squares.
\end{thmA}

A power difference set $K$ is also called a $(q, k, \lambda)$-difference set, where $\lambda$ is defined as in Theorem~\ref{thmA:1}. The situation described in Theorem~\ref{thmA:2} arises when $K$ is a difference set with $\lambda = 1$, in which case a finite projective plane can be constructed from $K$.

Consider the matrix $M = \left[ (i + q', j) \right]_{0 \leq i, j \leq \ell-1}$ with $q'=\frac{q-1}{2}$, and let $S$ be the submatrix obtained by removing the entries $\{(i + q', 0), (0 + q', j)\}_{0 \leq i, j \leq \ell-1}$ from $M$. Then $M = P A$ and $S = P' B$ for some orthogonal matrices $P$ and $P'$. In particular, both $M$ and $S$ are symmetric matrices. Thus, Theorem~\ref{thmA:1} can be applied to determine the spectra of $M$ and $S$, where we denote the spectrum of a matrix $T$ as $\Spec(T) = \{a_1^{(n_1)}, a_2^{(n_2)}, \ldots \}$, with $a_i$ being the distinct eigenvalues and $n_i$ their multiplicities. For somplicity, we write $a_i^{(1)} = a_i$.

\begin{thmA}
\label{thmA:3}
Suppose that $K$ is a $(q, k, \lambda)$-difference set. Then $$\Spec(M) = \left\{\frac{k + \sqrt{k^2 - 4 \lambda}}{2}, \frac{k - \sqrt{k^2 - 4 \lambda}}{2}, \sqrt{k - \lambda}^{\left(\frac{\ell}{2}-1\right)}, -\sqrt{k - \lambda}^{\left(\frac{\ell}{2}-1\right)} \right\}$$ and $$\Spec(S) = \left\{k - \lambda, \sqrt{k - \lambda}^{\left(\frac{\ell}{2}-1\right)}, -\sqrt{k - \lambda}^{\left(\frac{\ell}{2}-1\right)} \right\},$$ where the eigenvectors of $M$ and $S$ corresponding to $\pm \sqrt{k - \lambda}$ are of the form $(0, w)$ and $w$, respectively, for some vector $w \in \rr^{\ell-1}$ whose entries sum to zero.

Moreover, the minimal polynomials of $M$ and $S$ are: $$(x^2 - k x + \lambda)) (x^2 - (k - \lambda)) \quad \text{and} \quad (x - (k - \lambda)) (x^2 - (k - \lambda)),$$ respectively.
\end{thmA}

The preceding result can be used to compute the determinants of the cyclotomic matrix $A$ and the submatrix $B$ introduced in Theorem~\ref{thmA:1}.

\begin{cor*}
If $K$ is a $(q, k, \lambda)$-difference set, then $$\det(A) = - \lambda (k - \lambda)^{\frac{\ell}{2}-1} \quad \text{and} \quad \det(B) = (-1)^{\frac{\ell}{2}-1} (k - \lambda)^{\frac{\ell}{2}}.$$
\end{cor*}

The structure of this article is organized as follows. 

Section~\ref{sec2} introduces the notation and basic properties of cyclotomic numbers. In Section~\ref{sec3}, we present the connection between cyclotomic numbers and Schur rings, and prove Lemma~\ref{lemA:1}, which is labeled as Lemma~\ref{lem: product of cyclotomic matrices}. Section~\ref{sec: Inner product of columns in cyclotomic matrices} investigates inner products of the columns of the cyclotomic matrix, where we prove Theorem~\ref{thm: product of cyclotomic numbers}, a generalization of \cite[Theorem~6.1 and 6.2]{Sna}. 

In Section~\ref{sec: Power Difference Sets}, we begin with a brief introduction to the development of power difference sets. Subsection~\ref{subsec: Some necessary conditions and equivalent statements} presents several necessary conditions that must be satisfied if $K$ is a difference set, as well as part~(2) of Theorem~\ref{thmA:1}, which provides an equivalent reformulation of this property. Two additional necessary conditions are discussed in Subsection~\ref{subsec: Two additional properties}. In Subsection~\ref{subsec: Cyclotomic matrix to a power difference set}, we establish part~(3) of Theorem~\ref{thmA:1} and derive the expressions of $A^T A$ and $B^T B$, and includes Theorem~\ref{thmA:2} (labeled as Theorem~\ref{thm: if ell(k-lambda) is a square}. Subsection~\ref{subsec: Shifted cyclotomic matrices} provides the proof of Theorem~\ref{thmA:3} and its corollary. The expression of $A^T A$ for a {\it modified power difference set} is given in Subsection~\ref{subsec: Modified power difference set}. 

Finally, Section~\ref{sec: Concluding Remark} concludes with two problems related to power difference sets.


\section{Preliminaries}
\label{sec2}

Throughout this article, let $\ff_q$ denote a finite field with $q$ elements, where $q$ is a power of an {\it odd} prime $p$. For a divisor $\ell$ of $q-1$, denote by $K$ the multiplicative subgroup of $\ff_q^* = \ff_q \setminus \{0\}$ consisting of all $\ell$-th power. Then $\ell = \left[ \ff_q^* : K \right]$ and $q-1 = \ell k$, where $k = |K|$. For $0 \leq i, j \leq \ell-1$, the {\it cyclotomic numbers} with respect to a fixed generator $g$ of $\ff_q^*$ are defined as $$(i, j) = \left| (1 + g^i K) \cap g^j K \right|.$$ The indices $i$ and $j$ can be extended to arbitrary integers, and the values of the cyclotomic numbers are periodic modulo $\ell$; that is, $(i, j) = (i_1, j_1)$ whenever $i \equiv i_1 \pmod{\ell}$ and $j \equiv j_1 \pmod{\ell}$

We denote by $q'$ the unique integer between $0$ and $\ell-1$ such that $q' \equiv \frac{q-1}{2} \pmod{\ell}$. The assumption $q' = \frac{q-1}{2}$ in the introduction was made for convenience, but this choice does not affect the results. Since $q' \equiv -q' \pmod{\ell}$, the sign of $q'$ will often be omitted when it causes no ambiguity. Note that $-K = g^{\frac{q-1}{2}} K = g^{q'} K$. Moreover, $-K = K$ if and only if $k$ is even. Therefore, the value of $q'$ must be one of the following: \begin{align} q' = \left\{ \begin{array}{cl} 0 & \text{if $k$ is even,} \\ \frac{\ell}{2} & \text{if $k$ is odd.} \end{array} \right. \label{eq: value of q'} \end{align} 

The following elementary observation will be used frequently throughout this article.

\begin{lem}
\label{lem: basic properties for (i,j)}
Let $i, j \in \zz$. Then the following properties of cyclotomic numbers hold.
\begin{enumerate}
\item[(i)] 
$(i, j) = \left(j + q', i + q' \right)$.
\item[(ii)]
$(i, j) = (-i, j-i)$.
\item[(iii)]
$\sum_{t=0}^{\ell-1} (i, t) = \left\{ \begin{array}{cl} k-1 & \text{if } i \equiv q' \pmod{\ell}, \\ k & \text{otherwise.} \end{array} \right.$
\item[(iv)]
$\sum_{t=0}^{\ell-1} (t, j) = \left\{ \begin{array}{cl} k-1 & \text{if } j \equiv 0 \pmod{\ell}, \\ k & \text{otherwise.} \end{array} \right.$
\end{enumerate}
\end{lem}

\begin{proof}
We begin by noting that $-1 K = g^{q'} K$. To prove (i), consider the map $1 + g^i x \mapsto 1 - g^j y$ defined for $x, y \in K$ with $1 + g^i x = g^j y$. This gives a bijection between $(1+g^i K) \cap g^j K$ and $(1-g^j K) \cap (-g^i K)$, and the identity follows. To prove (ii), consider the map $1 + g^i x \mapsto (1 + g^i x) g^{-i} x^{-1} = 1 + g^{-i} x^{-1}$ which defines a bijection between $(1+g^i K) \cap g^j K$ and $(1+g^{-i} K) \cap g^{-i+j} K$. 

For (iii), observe that $$ \sum_{t=0}^{\ell-1} (i, t) =  \sum_{t=0}^{\ell-1} \left| (1 + g^i K) \cap g^t K \right| = \left| (1 + g^i K) \cap \ff_q^* \right|,$$ since $\ff_q^*$ decomposes as the disjoint union of the $g^t K$. The only missing element from $\ff_q^{*}$ is $0$, which lies in $1 + g^i K$ if and only if $g^i K = - K = g^{q'} K$, as desired. Finally, (iv) follows from (i) and (iii) because $$\sum_{t=0}^{\ell-1} (t, j) \overset{\text{(i)}}{=} \sum_{t=0}^{\ell-1} \left(j+q', t + q' \right) = \sum_{t=0}^{\ell-1} \left(j + q', t \right).$$ 
\end{proof}


The properties in (i) and (ii) will be used frequently throughout this article. When either identity is invoked, we will annotate the corresponding equality by placing $*$ and $**$ above it, respectively. For example, the following sequence of equalities illustrates their repeated application: \begin{align} \begin{split}(i, j) \overset{*}{=} (j+q', i+q') & \overset{**}{=} (-j + q', i-j) \\ & \overset{*}{=} (i-j + q', -j) \overset{**}{=} (j-i+q', -i+q') \overset{*}{=} (-i, j-i), \end{split} \label{eq: equalities of cyclotomic numbers} \end{align} and finally, $(-i, j-i) \overset{**}{=} (i, j)$.

The corollary below follows directly from the classification of $q'$ in~\eqref{eq: value of q'}.

\begin{cor}
\label{cor: basic properties for (i,j)}
If $k$ is even, then $(i, j) = (j, i)$ and $$\sum_{t=0}^{\ell-1} (i, t) = \left\{ \begin{array}{cl} k-1 & \text{if } i \equiv 0 \pmod{\ell}, \\ k & \text{otherwise.} \end{array} \right.$$ If $k$ is odd, then $(i, j) = \left(j + \frac{\ell}{2}, i + \frac{\ell}{2} \right)$ and $$\sum_{t=0}^{\ell-1} (i, t) = \left\{ \begin{array}{cl} k-1 & \text{if } i \equiv \frac{\ell}{2} \pmod{\ell}, \\ k & \text{otherwise.} \end{array} \right.$$
\end{cor}


We define the {\it cyclotomic matrix} as $$A = \left[ (i, j) \right]_{0 \leq i, j \leq \ell-1} = \left[ \begin{matrix} (0,0) & (0, 1) & \cdots & (0, \ell-1) \\ (1, 0) & (1,1) & \cdots & (1, \ell-1) \\ \vdots & \vdots & \ddots & \vdots \\ (\ell-1, 0) & (\ell-1, 1) & \cdots & (\ell-1, \ell-1) \end{matrix} \right].$$ This matrix has appeared in previous work as a convenient way to arrange the cyclotomic numbers in matrix form. 

\section{A connection to Schur rings}
\label{sec3}

Let $G$ be a finite group and let $\zz[G]$ denote its integral group ring. For any subset $P$ of $G$, we denote $P^{-1} = \{x^{-1} \mid x \in P\}$ and let $\overline{P} = \sum_{x \in P} x$ be the form sum over elements of $P$. Let $\mathcal{P} = \{P_1, P_2, \ldots, P_t\}$ be a partition of $G$, that is, $G$ is the disjoint union of the $P_i$. Consider the $\zz$-submodule $R \subseteq \zz[G]$ generated by $\overline{P_1}, \overline{P_2}, \ldots, \overline{P_t}$. The submodule $R$ is called a {\it Schur ring} (or {\it S-ring}) with respect to the partition $\mathcal{P}$ if $P_i^{-1} \in \mathcal{P}$ for every $1 \leq i \leq t$, and $R$ is closed under multiplication, i.e., a subring of $\zz[G]$.

The notion of Schur rings was introduced by I. Schur \cite{Sch} in order to generalize W. Burnside's work on permutation groups, though using terminology different from what is standard today. A systematic theory of such rings was later developed in H. Wielandt's book \cite{Wie2}. Schur rings have also found applications in the graph isomorphism problem---see the survey article \cite{MuzPon} for details. More recently, studies such as \cite{Mis} have investigated the number of partitions of a given cyclic $p$-group that generate Schur rings.

We now present a connection between cyclotomic numbers and Schur rings. Let $q = p^n$, and let $\beta = \{\beta_1, \beta_2, \ldots, \beta_n\}$ be a basis for $\ff_q$ over $\ff_p$. Consider the multiplicative elementary abelian group $G = \langle t_1 \rangle \times \langle t_2 \rangle \times \cdots \times \langle t_n \rangle$, and consider the group isomorphism $\varphi : (\ff_q, +) \to G$ given by $$x_1 \beta_1 + x_2 \beta_2 + \cdots + x_n \beta_n \quad \mapsto \quad t_1^{x_1} t_2^{x_2} \cdots t_n^{x_n}$$ for $x_1, x_2, \ldots, x_n \in \ff_p$. Then $$\mathcal{P} = \{\{1_G\}, \varphi(g^0 K), \varphi(g^1 K), \ldots, \varphi(g^{\ell-1} K)\}$$ forms a partition of $G$. Let $1$ denote the multiplicative identity in $\zz[G]$, and define $\alpha_i = \overline{\varphi(g^i K)}$ for $i \in \zz$. In particular, $\alpha_i = \alpha_{i'}$ whenever $i \equiv i' \pmod{\ell}$. Let $R$ be the $\zz$-submodule of $\zz[G]$ generated by $1, \alpha_0, \alpha_1, \ldots, \alpha_{\ell-1}$.

\begin{lem}
\label{lem: cyclotomic numbers and Schur ring}
The submodule $R$ is a Schur ring with respect to $\mathcal{P}$. Moreover for any integers $i, v$, the following identity holds: \begin{align} \alpha_i \alpha_v = \delta_{i, v+q'} + \sum_{j=0}^{\ell-1} (i-v, j-v) \alpha_j \label{eq: cyclotomic numbers and Schur ring} \end{align} where \begin{align*} \delta_{i, j} = \left\{ \begin{array}{cl} k & \text{if } i \equiv j \pmod{\ell}, \\ 0 & \text{otherwise.} \end{array} \right. \end{align*}
\end{lem}

\begin{proof}
For each $i = 0, 1, \ldots, \ell-1$, we have $(\varphi(g^i K))^{-1} = \varphi(-g^{i} K) = \varphi(g^{i + q'} K) \in \mathcal{P}$, so each inverse set remains within the partition, verifying the first condition of a Schur ring. To show that $R$ is closed under multiplication, it suffices to verify the identity~\eqref{eq: cyclotomic numbers and Schur ring} above. Fix integers $i$ and $v$. Since $\alpha_i \alpha_v \in \zz[G]$ and $\varphi: \ff_q \to G$ is an isomorphism, we may write $\alpha_i = \sum_{x \in g^i K} \varphi(x)$, $\alpha_v = \sum_{y \in g^v K} \varphi(y)$ and $\alpha_i \alpha_v = \sum_{z \in \ff_q} a_z \varphi(z)$ for some $a_z \in \zz$. For each $z$, the coefficient $a_z$ counts the number of ordered pairs $(x, y) \in g^i K \times g^v K$ such that $\varphi(x) \varphi(y) = \varphi(z)$, which is equivalent to $x + y = z$. Therefore, \begin{align*} a_z = \left| \{y \in g^v K \mid z - y \in g^i K \} \right| = \left| \{y \in g^v K \mid z y^{-1} - 1 \in g^{i-v} K \} \right|. \end{align*} If $z = 0$, the number of such $y$ is $k$ when $i \equiv v + q' \pmod{\ell}$, due to $-K = g^{q'} K$. Otherwise, the number is $\left| (1+g^{i-v} K) \cap (z g^{-v} K) \right|$. Now, for any fixed $j =0, 1, \ldots, \ell-1$, we have \begin{align*} \sum_{z \in g^j K} a_z \varphi(z) & = \sum_{z \in g^j K} \left| (1+g^{i-v} K) \cap (z g^{-v} K) \right| \varphi(z) \\ & = \sum_{z \in g^j K} (i-v, j-v) \varphi(z) \\ & = (i - v, j -v) \alpha_j. \end{align*} This completes the proof. 
\end{proof}


The proof of Lemma~\ref{lem: cyclotomic numbers and Schur ring} does not rely on the properties of cyclotomic numbers given in Lemma~\ref{lem: basic properties for (i,j)}. In fact, those properties can also be derived from the algebraic structure of the Schur ring $R$. For example, since $R$ is a commutative ring, the commutativity of $\alpha_i$ and $\alpha_0$ immediately yields property~(ii): $(i, j) = (-i, j-i)$. Furthermore, consider the scalar of $1$ in the product $(\alpha_i \alpha_0) \alpha_{j+q'}$, which equals $$\sum_{s=0}^{\ell-1} (i, s) k \delta_{s, j} = (i, j) k.$$ On the other hand, the scalar of $1$ in the expression $\alpha_i \left( \alpha_0 \alpha_{j+q'} \right)$ is $$\sum_{s=0}^{\ell-1} k \delta_{i, s+q'} \left( -j-q', s-j-q' \right) = k \left( -j-q', i - j \right).$$ By associativity of multiplication in $R$, we obtain $(i, j) = \left( -j-q', i - j \right)$, which, together with the commutativity of $\alpha_0$ and $\alpha_{j+q'}$, implies property~(i): $(i, j) = \left(j + q', i + q' \right)$. Finally, property~(iii) can be derived directly from the equation~\eqref{eq: cyclotomic numbers and Schur ring} by applying the {\it augmentation map}, which sends $\sum_{t \in G} a_t t \mapsto \sum_{t \in G} a_t$, $a_t \in \zz$, from $\zz[G]$ to $\zz$.

Motivated by the preceding arguments, it is natural to expect that further properties of cyclotomic numbers may be revealed through the algebraic structure of $R$. Let $[\alpha]$ denote the right regular representation of an element $\alpha$ of $R$ under the ordered basis $\{1, \alpha_0, \ldots, \alpha_{\ell-1}\}$. Equation~\eqref{eq: cyclotomic numbers and Schur ring} gives the explicit structure of the matrix $[\alpha_v]$, namely:
\begin{align*} [\alpha_v] = \left[ \begin{array}{c|c} 
0 & \mb{r}_{v} \\ \hline 
k \mb{c}_{v+q'} & A_v
\end{array} \right],
\end{align*}
where $A_v = [(i - v, j - v)]_{0 \leq i, j \leq \ell-1}$, $\mb{c}_{i}$ is the $\ell \times 1$ column vector with a $1$ in the $i$-th position and zeros elsewhere, $\mb{r}_j$ is the $1 \times \ell$ row vector with a $1$ in the $j$-th position and zeros elsewhere. All indices are considered modulo $\ell$, and range from $0$ to $\ell-1$. When $v=0$, the matrix $A_0$ is precisely the cyclotomic matrix $A$ defined earlier. In this case, we have 
\begin{align*} [\alpha_0] = 
\left[ \begin{array}{c|c} 
0 & \mb{r}_{0} \\ \hline 
k \mb{c}_{q'} & A
\end{array} \right]
\end{align*}
Furthermore, for $v = q'$, the matrix $[\alpha_{q'}]$ becomes
\begin{align*} 
[\alpha_{q'}] & = 
\left[ \begin{array}{c|c} 
0 & \mb{r}_{q'} \\ \hline 
k \mb{c}_{0} & A^T
\end{array} \right]
\end{align*}
where $A^T$ denotes the transpose of the cyclotomic matrix $A$, since the entries of $A_{q'}$ satisfy $(i - q', j - q') \overset{*}{=} (j, i)$.

Moreover, Equation~\eqref{eq: cyclotomic numbers and Schur ring} agives the following representation: \begin{align*} [\alpha_u \alpha_v] = [k \delta_{u, v+q'}] + \sum_{w=0}^{\ell-1} (u-v, w-v) [\alpha_w] \end{align*} which can be written as
\begin{align*}
\left[\begin{array}{c|c} 
k \delta_{u, v+q'} & \disp \sum_{w=0}^{\ell-1} (u-v, w-v) \mb{r}_{w} \\ \hline 
\disp k \sum_{w=0}^{\ell-1} (u-v, w-v) \mb{c}_{w+q'} & \disp k \delta_{u, v+q'} I_{\ell} + \sum_{w=0}^{\ell-1} (u-v, w-v) A_w 
\end{array} \right] 
\end{align*} 
where $I_{\ell}$ is the $\ell \times \ell$ identity matrix. On the other hand, we can compute the same product as 
\begin{align*} [\alpha_{u} \alpha_v] = [\alpha_{u}] \cdot [\alpha_v] = \left[ \begin{array}{c|c} 
k \mb{r}_{u} \mb{c}_{v+q'}  & \mb{r}_{u} A_v \\ \hline 
k A_{u} \mb{c}_{v+q'} & k E_{u+q', v} + A_{u} A_v 
\end{array} \right]\end{align*}
where $E_{i, j}$ denotes the matrix with a $1$ at the $(i, j)$-entry and zeros elsewhere. The $\ell \times \ell$ submatrix yields the following expression for $A_u A_v$, where our journey begins. 

\begin{lem}
\label{lem: product of cyclotomic matrices}
For $0 \leq u, v \leq \ell-1$, 
\begin{align} A_u A_v = k (\delta_{u, v+q'} I_{\ell} - E_{u+q', v}) + \sum_{w=0}^{\ell-1} (u-v, w-v) A_w. \label{eq: product of cyclotomic matrices} \end{align}
\end{lem}

\begin{remk}
\label{remk: conjugates of A_0}
Let $P_v$ be the permutation matrix whose $(i, j)$-entry is $1$ if $i - j \equiv v \pmod{\ell}$, and $0$ otherwise. Then $P_v^{-1} = P_{-v} = P_v^T$. In particular, we have the relation $A_v = P_v A_0 P_{v}^{-1}$. In this sense, Lemma~\ref{lem: product of cyclotomic matrices} describes the products of certain orthogonal conjugates of the cyclotomic matrix $A_0$. The matrix $P_{q'}$ will also play a role in Subsection~\ref{subsec: Shifted cyclotomic matrices}.
\end{remk}

From the identity $(u-v, w-v) \overset{**}{=} (v-u, w-u)$ and the commutativity of $\alpha_u$ and $\alpha_v$, it follows that \begin{align*} A_u A_v - A_v A_u = k (E_{v+q', u} - E_{u+q', v}). 
\end{align*} Also, since $(i - v + q', j - v + q') \overset{*}{=} (j - v, i - v)$, we find $A_{v+q'} = A_v^T$. Consequently, \begin{align*} A_v^T A_v - A_v A_v^T = k (E_{v+q', v+q'} - E_{v, v}). 
\end{align*} In other words, every cyclotomic matrix $A_0$ is ``almost'' a {\it normal matrix}. 

\begin{exmp}
\label{exmp: normality for 7^3}
Consider $\ff_{7^3} = \ff_7(g)$ with $g$ a root of the {\it Conway polynomial} $x^3 + 6 x^2 + 4$ over $\ff_7$. In particular, $g$ is also a generator of $\ff_{7^3}^*$. Let $\ell=6$, and consider the cyclotomic matrix $A_0$ associated with $\ell$ and $g$. Then $$A_0 = \left[ \begin{matrix} 7 & 12 & 12 & 14 & 6 & 6 \\ 12 & 9 & 9 & 6 & 12 & 9 \\ 9 & 9 & 12 & 6 & 9 & 12 \\ 7 & 12 & 9 & 7 & 12 & 9 \\ 12 & 6 & 9 & 12 & 9 & 9 \\ 9 & 9 & 6 & 12 & 9 & 12 \end{matrix} \right] \quad \text{and} \quad A_0^T A_0 - A_0 A_0^T = \left[ \begin{matrix} -57 & 0 & 0 & 0 & 0 & 0 \\ 0 & 0 & 0 & 0 & 0 & 0 \\ 0 & 0 & 0 & 0 & 0 & 0 \\ 0 & 0 & 0 & 57 & 0 & 0 \\ 0 & 0 & 0 & 0 & 0 & 0 \\ 0 & 0 & 0 & 0 & 0 & 0 \end{matrix} \right].$$
\end{exmp}

Note that $\tr(A_w) = k-1$ for each $w$, and $\sum_{w=0}^{\ell-1} (u-v, w-v) = k-\delta_{u-v, q'}$. Moreover, since $\delta_{u, v+q'} = \delta_{u-v, q'}$ and $\tr(E_{u+q',v}) = \delta_{u+q', v} = \delta_{u-v, q'}$, it follows from~\eqref{eq: product of cyclotomic matrices} that \begin{align} \begin{split} \tr(A_u A_v) & = k (\ell-1) \delta_{u-v, q'} + (k- \delta_{u-v, q'}) (k-1) \\ & = (q - 2k) \delta_{u-v, q'} + k (k-1) \end{split}. \label{eq: trace of A_u A_v} \end{align} Depending on the possible values of $q'$ as described in~\eqref{eq: value of q'}, we obtain the following result for the trace of the square of a cyclotomic matrix, classified by the parity of $k$.

\begin{prop}
\label{prop: trace of A^2}
\begin{align*} \tr(A_0^2) = 
\left\{ \begin{array}{ll} k (k-1) + q - 2k & \text{if $k$ is even,} \\ k (k-1) & \text{if $k$ is odd.} \end{array} \right. 
\end{align*}
\end{prop}

\begin{remk}
\label{remk: trace of A_0^3}
The trace of $A_0^2$ is independent of any cyclotomic number. Naturally, one may ask whether the trace of $A_0^3$ also enjoys such independence. By applying~\eqref{eq: product of cyclotomic matrices} and~\eqref{eq: trace of A_u A_v}, one arrives \begin{align*} \tr(A_0^3) = (0, q') (q-3k) + k^2 (k-1).
\end{align*} Although this expression does depend on the cyclotomic number $(0, q')$, it is worth noting that is does not depend on the parity of $k$.
\end{remk}

The sum of the squares of all cyclotomic numbers can be obtained from the identity $\tr(A_0^T A_0) = \tr(A_{q'} A_0)$, which evaluates to the expression $q + k (k-3)$. This offers an alternative proof of the result \cite[Lemma~2.2]{BetHirKomMun}, where the same quantity is utilized in the computation of the variance of cyclotomic numbers. 

\begin{lem}[{\cite[Lemma~2.2]{BetHirKomMun}}]
\label{lem: Lemma 2.2 of BetHirKomMun}
$\disp\sum_{0 \leq i, j \leq \ell-1} (i, j)^2 = q + k(k-3)$.
\end{lem}

This total sum of squares can be interpreted as the aggregate of standard inner products taken over each column of the cyclotomic matrix with itself. In the following section, we turn to the inner product between distinct columns of the cyclotomic matrix.

\section{Inner product of columns in cyclotomic matrices}
\label{sec: Inner product of columns in cyclotomic matrices}

We begin by rewriting the expression for the product of $A_u, A_v$ in transposed form. Using the identity $A_u^T = A_{u+q'}$ and the fact that $\delta_{u+q', v+q'} = \delta_{u,v}$, equation~\eqref{eq: product of cyclotomic matrices} becomes \begin{align} A_u^T A_v = k (\delta_{u, v} I_{\ell} - E_{u, v}) + \sum_{w=0}^{\ell-1} (u-v+q', w-v) A_w. \label{eq: product of cyclotomic matrices 2} \end{align} 
This formulation reflects how the product $A_u^T A_v$ encodes standard inner products between the column vectors of $A_u$ and those $A_v$. From this, we can derive the following identity describing inner product relationships between columns in cyclotomic matrices. 

\begin{thm}
\label{thm: product of cyclotomic numbers}
Let $i, j, u, v \in \zz$. Then
\begin{align*} \sum_{w=0}^{\ell-1} (w-u, i-u) (w-v, j-v) = k( \delta_{i,j} \delta_{u, v} - \delta_{i, u} \delta_{j, v}) + \sum_{w=0}^{\ell-1} (w-v, u-v) (w-j, i-j) 
\end{align*}
\end{thm}

\begin{proof}
First of all, assume $0 \leq i, j, u, v \leq \ell-1$. The $(i,j)$-entry of $A_u^T A_V$ is given by $$\sum_{w=0}^{\ell-1} (i-u+q', w-u+q') (w-v, j-v) \overset{*}{=} \sum_{w=0}^{\ell-1} (w-u, i-u) (w-v, j-v).$$ On the other hand, we observe that $(u-v+q', w-v) \overset{*}{=} (w-v+q', u-v)$ and $(i-w, j-w) \overset{*}{=} (j-w+q', i-w+q') \overset{**}{=} (w-j+q', i-j)$. Thus, the $(i, j)$-entry of $\sum_{w=0}^{\ell-1} (u-v+q', w-v) A_w$ becomes \begin{align*} \sum_{w=0}^{\ell-1} (u-v+q', w-v) (i-w, j-w) & \overset{*}{=} \sum_{w=0}^{\ell-1} (w-v+q', u-v) (w-j+q', i-j) \\ & = \sum_{w=0}^{\ell-1} (w-v, u-v) (w-j, i-j), \end{align*} where the last equality holds by shifting $w$ to $w+q'$. 

Now consider the term $\delta_{u, v} I_{\ell} - E_{u, v}$. If $(i, j) = (u, v)$, its $(i, j)$-entry equals $\delta_{u,v} - 1$. Otherwise, the entry $(i, j) \neq (u, v)$ is $\delta_{i, j} \delta_{u, v}$, the $(i, j)$-entry of $\delta_{u, v} I_{\ell}$. The result is obtained by~\eqref{eq: product of cyclotomic matrices 2}. Finally, since each step holds under modular arithmetic, the identity extends to arbitrary integers $i, j, u, v$ as well. 
\end{proof}

\begin{remk}
\label{remk: Snapper}
Theorem~\ref{thm: product of cyclotomic numbers} generalizes Theorem~6.1 and~6.2 in E. Snapper's work \cite{Sna}, where the author expressed uncertainty regarding their validity when $q$ is not prime. More precisely, the number $c_{s,t}$ in his notation corresponds exactly to the cyclotomic number $(s+q', t)$, and satisfies the symmetry $c_{s,t} = c_{t,s}$. 

When $i = u \neq 0$ and $j = v = 0$, the setting of his Theorem~6.1 arises as a special case of our result. 

To recover his Theorem~6.2, consider the case $u \neq v$ and $(i, j) \not\equiv (u, v)$ modulo $\ell$, so that $\delta_{u, v} = \delta_{i, u} \delta_{j,v} = 0$. Using the identity $(w-v, j-v) \overset{*}{=} (j-v+q', w-v+q') \overset{**}{=} (v-j +q', w-j) \overset{*}{=} (w-j+q', v-j)$, define $a = i-u$, $b = v-j$ and $m = u-j+q'$. Then we obtain $$\sum_{w=0}^{\ell-1} (w-u, i-u) (w-v, j-v) = \sum_{w=0}^{\ell-1} (w-u, a) (w-j+q', b) = \sum_{w=0}^{\ell-1} (w, a) (w+m, b),$$ after shifting the index $w$ to $w+q'$. On the other hand, observe that $(w-v, u-v) \overset{**}{=} (v-w, u-w) \overset{*}{=} (u-w+q', v-w+q') \overset{**}{=} (w-u+q', v-u)$. Set $a' = v-u$ and $b' = i-j$, and then $$ \sum_{w=0}^{\ell-1} (w-v, u-v) (w-j, i-j) = \sum_{w=0}^{\ell-1} (w-u+q', a') (w-j, b') = \sum_{w=0}^{\ell-1} (w, a') (w+m, b'),$$ 
after shifting $w$ to $w+u+q'$. Therefore, Theorem~\ref{thm: product of cyclotomic numbers} implies the identity \begin{align} \sum_{w=0}^{\ell-1} (w, a) (w+m, b) = \sum_{w=0}^{\ell-1} (w, a') (w+m, b'). \label{eq: Snapper's Theorem 6.2} \end{align} This is exactly the statement of Snapper's Theorem~6.2. In this setup, the parameters satisfy $a' \not\equiv 0 \pmod{\ell}$, $a+b = a'+b'$ and $m \equiv q' + b - a'$. 

Notably, Snapper's original hypothesis requires $a \not\equiv 0 \pmod{\ell}$, whereas our result only assume $(a,b) \neq (0,0)$, which is guaranteed when $(i, j) \neq (u,v)$. As an exmple, equation~\eqref{eq: Snapper's Theorem 6.2} holds for $a=0$, $b = 3$, $a' = 1$, and $b' = 2$, as verified by Example~\ref{exmp: normality for 7^3}.

The matrix $[c_{s,t}]=[(s+q', t)]$ introduced in Snapper's formulation will also appear again in Subsection~\ref{subsec: Shifted cyclotomic matrices}.
\end{remk}

We now derive the standard inner product between any two columns of the cyclotomic matrix $A_0$. Remarkably, the square sum of each column depends only on the entries in the first column.

\begin{prop}
\label{prop: dot product of columns}
Let $1 \leq i \leq \ell-1$. Then $$\sum_{w=0}^{\ell-1} (w, i)^2 = k + \sum_{w=0}^{\ell-1} (w, 0) (w - i, 0).$$ For $0 \leq i \neq j \leq \ell-1$, we have $$\sum_{w=0}^{\ell-1} (w, i) (w, j) = \sum_{w=0}^{\ell-1} (w, 0) (w - j, i-j).$$
In addition, if $k$ is odd (equivalently, $q' \neq 0$), then $$\sum_{w=0}^{\ell-1} (w, q')^2 = k + \sum_{w=0}^{\ell-1} (w, 0)^2,$$ and for $(i, j) \not\equiv (0, 0), (q', q') \pmod{\ell}$, we have $$\sum_{w=0}^{\ell-1} (w, i) (w, j) = \sum_{w=0}^{\ell-1} (w, i+q') (w, j+q').$$ 
\end{prop}

\begin{proof}
Applying $u = v = 0$ in Theorem~\ref{thm: product of cyclotomic numbers}, we have $$\sum_{w=0}^{\ell-1} (w, i) (w, j) = k (\delta_{i, j} - \delta_{i,0} \delta_{j, 0}) + \sum_{w=0}^{\ell-1} (w, 0) (w - j, i-j),$$ from which the first two equalities follows directly. Now suppose $k$ is odd. Since $(w - q', 0) \overset{**}{=} (q'-w, q'-w) \overset{*}{=} (-w, -w) \overset{**}{=} (w, 0)$, the sum of $(w,q')^2$ follows. For the final equality, we observe that $(w - j, i - j) \overset{**}{=} (j - w, i - w) \overset{*}{=} (i+q' - w, j+q' - w) \overset{**}{=} (w - (i+q'), j-i)$. Thus, \begin{align*} & \mbox{} \quad \sum_{w=0}^{\ell-1} (w, i) (w, j) - \sum_{w=0}^{\ell-1} (w, j+q') (w, i+q') \\ & = k (\delta_{i,j} - \delta_{i, 0} \delta_{j, 0}) - k (\delta_{j+q', i+q'} - \delta_{j+q',0} \delta_{i+q',0}) \\ & = k( -\delta_{i,0} \delta_{j,0} + \delta_{i, q'} \delta_{j, q'}).\end{align*} The result follows. 
\end{proof}




\begin{exmp}
Consider the field $\ff_{131}$ with $\ell=10$ and $k = 13$. The cyclotomic matrix $A_0$, constructed with respect to $\ell$ and a primitive root $2$, is given by 
$$A_0 = \left[\begin{matrix} 2 & 1 & 0 & 0 & 2 & 0 & 2 & 0 & 4 & 2 \\ 1 & 0 & 2 & 0 & 2 & 2 & 2 & 2 & 1 & 1 \\ 1 & 1 & 2 & 3 & 1 & 0 & 2 & 0 & 0 & 3 \\ 2 & 1 & 1 & 1 & 1 & 4 & 1 & 0 & 0 & 2 \\ 0 & 1 & 1 & 1 & 1 & 2 & 1 & 3 & 2 & 1 \\ 2 & 1 & 1 & 2 & 0 & 2 & 1 & 1 & 2 & 0 \\ 1 & 2 & 1 & 3 & 2 & 1 & 0 & 1 & 1 & 1 \\ 1 & 1 & 4 & 1 & 0 & 0 & 2 & 2 & 1 & 1 \\ 2 & 3 & 1 & 0 & 2 & 0 & 0 & 3 & 1 & 1 \\ 0 & 2 & 0 & 2 & 2 & 2 & 2 & 1 & 1 & 1 \end{matrix}\right].$$ To help verify the validity of Proposition~\ref{prop: dot product of columns}, we provide the corresponding {\it Gram matrix} $A_0^T A_0$, which represents the inner products between columns of $A_0$: $$A_0^T A_0 = \left[\begin{matrix} 20 & 16 & 15 & 13 & 15 & 15 & 14 & 13 & 17 & 16 \\ 16 & 23 & 14 & 18 & 19 & 14 & 13 & 19 & 16 & 16 \\ 15 & 14 & 29 & 17 & 12 & 13 & 19 & 20 & 12 & 17 \\ 13 & 18 & 17 & 29 & 15 & 17 & 16 & 12 & 12 & 18 \\ 15 & 19 & 12 & 15 & 23 & 16 & 16 & 17 & 18 & 18 \\ 15 & 14 & 13 & 17 & 16 & 33 & 16 & 15 & 13 & 15 \\ 14 & 13 & 19 & 16 & 16 & 16 & 23 & 14 & 18 & 19 \\ 13 & 19 & 20 & 12 & 17 & 15 & 14 & 29 & 17 & 12 \\ 17 & 16 & 12 & 12 & 18 & 13 & 18 & 17 & 29 & 15 \\ 16 & 16 & 17 & 18 & 18 & 15 & 19 & 12 & 15 & 23 \end{matrix}\right].$$ 
\end{exmp}

\begin{remk}
Suppose $k$ is odd. According to Proposition~\ref{prop: dot product of columns}, one has $$\sum_{w=0}^{\ell-1} (w, j)^2 = \sum_{w=0}^{\ell-1} (w, j+q')^2$$ for $j = 1, \ldots, q'-1$, provided $\ell \neq 2$. For instance, when $\ell = 4$, one has $q' = 2$, and observe that $(w, 1) \overset{**}{=} (-w, 1-w) \overset{*}{=} (3-w, 2-w) \overset{**}{=} (w-3, -1) = (w-3, 1 + q')$, where the last equality uses the congruence $-1 \equiv 1+q' \pmod{4}$. It follows that the multiset of the entries in the column vector $[(w, 1)]_{0 \leq w \leq \ell-1}$ is a permutation of the entries in the column vector $[(w, 1+q')]_{0 \leq w \leq \ell-1}$. Hence, the identity above holds trivially in this case. 

Through computational verification, the same phenomenon appears to hold when $\ell = 6$ and $\ell = 8$. However, we currently lack a conceptual explanation for this behavior. We therefore propose the following question for interested reader.

\begin{ques}
For which value $\ell$ with odd $k$, are the entries of the cyclotomic matrix column indexed by $j$ a permutation of those indexed by $j+q'$, for each $j = 1, \ldots, q'-1$?
\end{ques}
\end{remk}

\section{Power Difference Sets}
\label{sec: Power Difference Sets}

One of our motivations for studying cyclotomic numbers originates from a classical problem in the theory of difference sets. Let $G$ be a finite group. A subset $H$ of $G$ is called a {\it difference set} if there exists a positive integer $\lambda$ such that every nontrivial element of $G$ can be expressed as $x y^{-1}$ for exactly $\lambda$ distinct pairs of $(x, y)$ with $x, y \in H$. Such a set is referred to as a {\it $(v, k, \lambda)$-difference set}, where $v = |G|$ and $k = |H|$. The concept of difference sets has its origins in the study of finite projective geometry. It was later recognized to be closely connected with combinatorial design theory, especially with balanced incomplete block designs (BIBDs). We refer the reader to the surveys \cite{MHall2, Jungnickel-Pott, Momihara-Wang-Xiang} for further background and developments on this topic.

In our setting, we take $G$ to be the additive group of the finite field $\ff_q$. A subset $K$ of $\ff_q$ that is both a multiplicative subgroup of $\ff_q^*$ and a difference set with respect to the additive group structure is called a {\it power difference set}. In the literature, such sets are also known as {\it power residue difference set} (particularly when $q$ is prime)~\cite{Leh3}, or {\it cyclotomic difference set}~\cite{BXia}. The name ``power difference set'' arises from the fact that $K$ consists precisely of all $\ell$-th powers in $\ff_q^*$, where $\ell$ is the index $[\ff_q^* : K]$. In what follows, we will consider only difference sets of this type.

Let $p$ be the characteristic of $\ff_q$. It was known from Paley \cite{Paley} that the subgroup of squares in $\ff_p^*$ (that is, the case $\ell=2$) forms a difference set if and only if $p \equiv 3 \pmod{4}$. Later, Chowla~\cite{Chowla} established that the subgroup of fourth powers (i.e., $\ell=4$) forms a difference set whenever $p = 1 + 4 t^2$ for some odd integer $t$. 

In 1953, Emma Lehmer established the following theorem in connection with cyclotomic numbers. While her original result was stated for $\ff_p$, where $p$ is an odd prime, her argument applies equally well to any finite field $\ff_q$. For convenience, we restate her result and include a proof. The proof relies on the following lemma, where $g$ denotes a generator of $\ff_q^*$, as defined earlier.

\begin{lem}[{\cite[Lemma~I]{Leh3}}]
\label{lem: Lemma I in Leh3}
For $0 \leq j \leq \ell-1$, the cyclotomic number $(0, j)$ is odd if and only if $2 \in g^j K$.
\end{lem}

\begin{proof}
Let $x \in K$. Then $1+x \in g^j K$ if and only if $1 + x^{-1} \in g^j K$, since $1+x^{-1} = (1+x) x^{-1}$, and $K$ is closed under multiplication. Therefore, the set $(1+K) \cap g^j K$ contains both $1 + x$ and $1 + x^{-1}$ for each such $x$, provided that if $(0, j) \neq 0$. The only situation in which $x = x^{-1}$ occurs is when $x=1$, since $g^j K$ does not contain $0$. 
\end{proof}

We now rephrase \cite[Theorem~I, II and III]{Leh3} into the following formulation.

\begin{thm}[{\cite{Leh3}}]
\label{thm: Theorem I, II, III in Leh3}
Let $K$ be a subgroup of $\ff_q^*$ with index $\ell \geq 2$, and let $k = |K|$. Then $K$ is a difference set of $\ff_q$ if and only if $(0,0) = (1, 0) = \cdots = (\ell-1, 0)$. In this case, $k$ is odd, $\ell$ is even and $(0, 0) = \frac{k-1}{\ell}$.
\end{thm}

\begin{proof}
Let $z \in \ff_q^*$, and suppose $z \in g^i K$. Then \begin{align*}\left| \{(x, y) \in K \times K \mid x-y = z \} \right| & = \left| \{y \in K \mid y + z \in K \} \right| \\ & = \left| \{y \in K \mid 1 + y^{-1} z \in y^{-1} K \} \right| \\ & = \left| (1 + g^i K) \cap K \right| \\ & = (i, 0). \end{align*}
Therefore, $K$ is a difference set of $\ff_q$ if and only if all values $(i, 0)$ are equal, which we denote by a constant $\lambda$. By part~(iv) of Lemma~\ref{lem: basic properties for (i,j)}, we obtain $\lambda = (k-1) / \ell$. 

Now suppose that $k$ is even. Then $q' = 0$, and so $(i, 0) \overset{*}{=} (0, i)$ for each $i$. In particular, all of the values $(0, i)$ would share the same parity, which contradicts to Lemma~\ref{lem: Lemma I in Leh3}. Hence, $k$ must be odd, and since $q-1 = \ell k$, it follows that $\ell$ must be even. 
\end{proof}

For the case where $q = p$, Lehmer utilized the formulas for cyclotomic numbers originally derived by Gauss to give an alternative proof of Paley's result when $\ell = 2$, as well as the necessity of Chowla's condition when $\ell = 4$. She further showed the nonexistence of power difference sets for $\ell = 6$, and established a necessary and sufficient condition for the existence of such sets for $\ell = 8$, based on formulas of cyclotomic numbers obtained by L.E. Dickson \cite{Dic}. Since then, much effort has been devoted to deriving explicit formulas for cyclotomic numbers. However, these formulas become increasingly intricate and difficult to determine as $\ell$ grows. To date, the nonexistence of power difference sets have been established using formulas for $\ell = 10$~\cite{Whi3}, 12~\cite{Whi4}, 14~\cite{Mus}, 16~\cite{Whi2, Eva}, 18~\cite{BauFre}, 20~\cite{MusWhi, Eva4}, and $24$~\cite{Eva3, EvaVan}. Instead of developing such formulas for each individual $\ell$, a computational method applicable to general $\ell$ was proposed in \cite{BXia}, which showed the nonexistence for $\ell=22$. This method was further extended  in \cite{BXia2} to establish the nonexistence for $\ell=26, 32$. These results motivate the following conjecture.

\begin{conjecture}
\label{conj: power difference set}
If the multiplicative subgroup of $\ell$-th powers forms a difference set of $\ff_p$, then $\ell = 2, 4, 8$.
\end{conjecture}
 
We remark here that the nonexistence of power difference set in a finite field of characteristic $2$ has been established in \cite[Theorem~4.1]{BXia}.

\subsection{Some necessary conditions and equivalent statements}
\label{subsec: Some necessary conditions and equivalent statements}

\mbox{}

We present several congruence conditions on $q$ and $\ell$ in relation to power difference sets, particularly focusing on the parity of $\lambda$. Recall that $2$ is a square in $\ff_q$ if and only if $q \equiv \pm 1 \pmod{8}$. 

\begin{lem}
\label{lem: difference set with odd lambda}
If $K$ is a $(q, k, \lambda)$-difference set with odd $\lambda$, then one of the following holds: 
\begin{enumerate}
\item[(i)] $\ell \equiv 0 \pmod{8}$ and $q \equiv k \equiv 1 \pmod{8}$; 
\item[(ii)] $\ell \equiv 2 \pmod{8}$, $q \equiv 7 \pmod{8}$ and $k \equiv 2 \lambda + 1 \pmod{8}$.
\end{enumerate}
\end{lem}

\begin{proof}
Since $\lambda$ is odd, $2$ must be a square in $\ff_q$, which implies $q \equiv \pm 1 \pmod{8}$. On the other hand, $q = \ell k + 1 = \ell (\lambda \ell + 1) + 1 \equiv \ell^2 + \ell + 1 \pmod{8}$, from which we conclude $\ell \equiv 0 \pmod{8}$ when $q \equiv 1 \pmod{8}$, and $\ell \equiv 2 \pmod{8}$ when $q \equiv -1 \pmod{8}$. The desired congruence for $k = \lambda \ell + 1$ follows accordingly.  
\end{proof}

The congruence of $\lambda$ modulo $4$ in case~(i) of the preceding lemma is further clarified in Proposition~\ref{prop: odd lambda}.

The following result is inspired by {\cite[Theorem~4, remarked by A. Schinzel]{Mus}}.

\begin{lem}
\label{lem: parity of lambda}
Suppose $K$ is a difference set with $\ell \equiv 2 \pmod{4}$. Then $\lambda \equiv \frac{k-1}{2} \pmod{2}$.
\begin{enumerate}
\item[(i)]
If $\ell \equiv 2 \pmod{8}$, then $\lambda$ is odd (resp. even) if and only if $q \equiv 7 \pmod{8}$ (resp. $q \equiv 3 \pmod{8}$). 
In particular, $\lambda$ is odd if and only if $2$ is a square in $\ff_q^{\times}$.
\item[(ii)]
If $\ell \equiv 6 \pmod{8}$, then $\lambda$ is even, $k \equiv 1 \pmod{4}$, $q \equiv 7 \pmod{8}$ and $2$ is a square in $\ff_q^{\times}$. Moreover, $(0, j)$ is odd for some even integer $j$ with $0 < j \leq \ell-1$.
\end{enumerate}
\end{lem}

\begin{proof}
First of all, $k = \lambda \ell + 1$ and $q = \ell k + 1 = \lambda \ell^2 + \ell + 1$. When $\ell \equiv 2 \pmod{4}$, one has $k \equiv 2 \lambda + 1 \pmod{4}$ and $q \equiv 4 \lambda + \ell + 1 \pmod{8}$. 
In particular, this implies that $\frac{k-1}{2} \equiv \lambda \pmod{2}$. Hence, if $\lambda$ is odd (resp. even), then $q \equiv \ell + 5 \pmod{8}$ (resp. $q \equiv \ell + 1 \pmod{8}$), and part~(i) follows. 

For the case $\ell \equiv 6 \pmod{8}$, we know from Lemma~\ref{lem: difference set with odd lambda} that $\lambda$ must be even. In this case, we have $q \equiv \ell + 1 \equiv 7 \pmod{8}$. Thus, $2 \not\in K$, but $2$ is a square in $\ff_q^{\times}$. Write $2 \in g^j K$ for $0 < j \leq \ell-1$, and $(0, j)$ is odd by Lemma~\ref{lem: Lemma I in Leh3}. Since all elements of $K$ are squares, $j$ must be even. 
\end{proof}

It is somewhat unexpected that the minimality of the sum of squares of cyclotomic numbers $(i, 0)$ for $i = 0, 1, \ldots, \ell-1$ characterizes when $K$ forms a difference set, as shown below.

\begin{lem}
\label{lem: sum of squares of (i,0)}
$$\sum_{i=0}^{\ell-1} (i, 0)^2 \geq \frac{\left(k-1\right)^2}{\ell}.$$ Equality holds if and only if $K$ is a difference set of $\ff_q$.
\end{lem}

\begin{proof}
By the Cauchy-Schwarz inequality, one has $$\sum_{i=0}^{\ell-1} (i, 0)^2 \cdot \sum_{i=0}^{\ell-1} 1^2 \geq \left( \sum_{i=0}^{\ell-1} (i, 0) \cdot 1 \right)^2 = (k-1)^2.$$ Equality holds if and only if all $(i, 0)$ are equal. The result then follows from Theorem~\ref{thm: Theorem I, II, III in Leh3}.
\end{proof}

In contrast to Lehmer's original observation, the cyclotomic numbers $(0, j), (1, j), \ldots, (\ell-1, j)$ for $j \neq 0$ can also indicate whether $K$ forms a difference set, through a similar sum-of-squares perspective. Here, $q'$ is defined as in Section~\ref{sec2}.

\begin{lem}
\label{lem: sum of squares of (i,j)}
Let $k$ be odd. For $1 \leq j \leq \ell-1$, we have $$\sum_{i=0}^{\ell-1} (i, j)^2 \leq \sum_{i=0}^{\ell-1} (i, q')^2.$$ Equality holds for some $j$ with $\gcd(j, \ell) = 1$ if and only if $K$ is a difference set. 
\end{lem}

\begin{proof}
According to Proposition~\ref{prop: dot product of columns}, $$\sum_{i=0}^{\ell-1} (i, j)^2 = k + \sum_{i=0}^{\ell-1} (i, 0) (i-j, 0),$$ and when $k$ is odd, $$\sum_{i=0}^{\ell-1} (i, q')^2 = k + \sum_{i=0}^{\ell-1} (i, 0)^2.$$ The inequality then follows from the Cauchy-Schwarz inequality that $$\left( \sum_{i=0}^{\ell-1} (i, 0) (i-j, 0) \right)^2 \leq \sum_{i=0}^{\ell-1} (i, 0)^2 \cdot \sum_{i=0}^{\ell-1} (i-j, 0)^2 = \left( \sum_{i=0}^{\ell-1} (i, 0)^2 \right)^2.$$ If $K$ is a difference set, then $k$ is odd and all $(i, 0)$ are equal by Theorem~\ref{thm: Theorem I, II, III in Leh3} (Lehmer). Thus, equality holds.

Conversely, suppose equality holds for some $j$ with $\gcd(j, \ell)=1$. Then, equality in the Cauchy-Schwarz inequality implies the existence of a constant $c$ such that $(i, 0) = c (i-j, 0)$ for each $i = 0, 1, \ldots, \ell-1$. Comparing the total sums of $(i, 0)$ and $(i-j, 0)$, we obtain $c=1$. Hence, $(i, 0) = (i-j, 0)$ for all $i$. Replacing $i$ by $i+j$ repeatedly yields $(0, 0) = (j, 0) = (2 j, 0) = \cdots = (t j, 0) = \cdots$ for all $t \in \nn$. Since $\gcd(j, \ell) = 1$, the sequence $t j$ modulo $\ell$ cycles through all residue classes modulo $\ell$. Therefore, all $(i, 0)$ are equal, and $K$ is a difference set by Theorem~\ref{thm: Theorem I, II, III in Leh3} (Lehmer).
\end{proof}

\subsection{Cyclotomic matrix to a power difference set}
\label{subsec: Cyclotomic matrix to a power difference set}

\mbox{}

The cyclotomic matrix naturally encodes the information of cyclotomic numbers. However, its matrix-theoretic properties have rarely been studied. We are going to utilize the perspective of matrix theory to study power difference sets. 

Recall the matrices $A_v$ and $E_{i,j}$ defined in Section~\ref{sec3}, and $A = A_0$. 

\begin{thm}
\label{thm: A^T A for a difference set K}
Let $K$ be the subgroup of all $\ell$-th powers in $\ff_q^{\times}$, where $\ell \mid q-1$, and let $A$ be the cyclotomic matrix with respect to $\ell$. Then $K$ is a difference set of $\ff_q$ if and only if $|K|$ is odd and the main diagonal entries of $A^T A$ has the form $(a, b, \ldots, b)$ for some $a, b \in \nn$. 

In this case, $$A^T A = \lambda k J_{\ell} + (k-\lambda) I_{\ell} - k E_{0,0},$$ where $\lambda = \frac{k-1}{\ell}$ and $J_{\ell}$ is the $\ell \times \ell$ all-one matrix.
\end{thm}

\begin{proof}
Suppose that $K$ is a $(q, k, \lambda)$-difference set. According to~\eqref{eq: product of cyclotomic matrices 2}, one has $$A^T A = k (I_{\ell} - E_{0,0}) + \sum_{w=0}^{\ell-1} (q', w) A_w.$$ Note that $(q', w) \overset{*}{=} (w + q', 0) = \lambda$. Moreover, the $(i, j)$-entry of $\sum_{w=0}^{\ell-1} A_w$ is $$\sum_{w=0}^{\ell-1} (i-w, j-w) \overset{**}{=} \sum_{w=0}^{\ell-1} (w-i, j-i) = \left\{ \begin{array}{cl} k-1 & \text{if $i=j$,} \\ k & \text{otherwise,} \end{array} \right.$$ by Lemma~\ref{lem: basic properties for (i,j)}. Thus, $\sum_{w=0}^{\ell-1} A_w = k J_{\ell} - I_{\ell}$, and the claimed expression for $A^T A$ follows. 

In general, the $j$-th diagonal entry of $A^T A$ is given by $\sum_{i=0}^{\ell-1} (i, j)^2$. If these entries take the form $(a, b, \ldots, b)$ and $k$ is odd, then Lemma~\ref{lem: sum of squares of (i,j)} applied to $j=1$ implies that $K$ must be a difference set. 
\end{proof}

The result above is independent of the choice of generator for $\ff_q^{\times}$. The condition that $|K|$ is odd is indeed necessary. For instance, when $(q, \ell) = (9, 4)$ so that $|K| = 2$, the subgroup $K$ is not a difference set, yet the main diagonal of $A^T A$ is $(1, 2, 2, 2)$.

Let us know consider a submatrix of the cyclotomic matrix. Define $B$ to be the $(q',0)$-minor of $A$, i.e., the matrix obtained by removing the $q'$-th row and the first column (indexed by $0$) from $A$. Then the entries of $B$ are $(i, j)$ for $0 \leq i \neq q' \leq \ell-1$ and $1 \leq j \leq \ell-1$. In particular, the $(i, j)$-entry of $B^T B$, for $1 \leq i, j \neq q' \leq \ell-1$, is given by $$\sum_{\substack{t=0, \\ t\neq q'}}^{\ell-1} (t, i) (t, j) = \sum_{t=0}^{\ell-1} (t, i) (t, j) - (q', i) (q', j),$$ which is precisely the $(i, j)$-entry of $A^T A$ minus $\lambda^2$. This yields the following corollary.

\begin{cor}
\label{cor: A^T A for a difference set K}
If $K$ is a $(q, k, \lambda)$-difference set, then \begin{align*} B^T B = (k - \lambda) (\lambda J_{\ell-1} + I_{\ell-1}). \end{align*} 
\end{cor}

\begin{exmp}
\label{exmp: B^TB to ell=2}
Consider the case $\ell = 2$. By Paley's result (works also for a general finite field), $K$ is a difference set if and only if $q \equiv 3 \pmod{4}$ (Paley). In this situation, it follows from Lemma~\ref{lem: basic properties for (i,j)} that $(0, 0) = (1, 0) = (1, 1) = \frac{k-1}{2} = \lambda$, and $(0,1) = \frac{k+1}{2}$. Hence, $B$ is the $1 \times 1$ matrix $B = [(0,1)]$. Moreover, $(0,1)^2 = (\lambda + 1)^2 = (k - \lambda) (\lambda + 1)$. 
\end{exmp}

\begin{exmp}
\label{exmp: B^TB to ell=4 and 8}
Consider $q = p$, an odd prime.

(I) 
Suppose $\ell = 4$. According to Chowla's result, $K$ is a difference set when $p = 1 + 4 t^2$ for odd $t > 1$. Such primes includes $37, 101, 197, 677, 2917, 4357, 5477, \ldots$. Below is the matrix $B$ for the first three primes. The primitive root is chosen as $2$ in each case so that the reader may verify $B^T B$: $$\begin{array}{|c|c|c|c|} \hline p & 37 & 101 & 197 \\ \hline k & 9 & 25 & 49 \\ \hline \lambda & 2 & 6 & 12 \\ \hline A & \left[\begin{matrix} 2 & 1 & 2 & 4 \\ 2 & 2 & 4 & 1 \\ 2 & 2 & 2 & 2 \\ 2 & 4 & 1 & 2 \end{matrix}\right] & \left[\begin{matrix} 6 & 9 & 6 & 4 \\ 6 & 6 & 4 & 9 \\ 6 & 6 & 6 & 6 \\ 6 & 4 & 9 & 6 \end{matrix}\right] & \left[\begin{matrix} 12 & 9 & 12 & 16 \\ 12 & 12 & 16 & 9 \\ 12 & 12 & 12 & 12 \\ 12 & 16 & 9 & 12 \end{matrix}\right] \\ \hline B & \left[\begin{matrix} 1 & 2 & 4 \\ 2 & 4 & 1 \\ 4 & 1 & 2 \end{matrix}\right] & \left[\begin{matrix} 9 & 6 & 4 \\ 6 & 4 & 9 \\ 4 & 9 & 6 \end{matrix}\right] & \left[\begin{matrix} 9 & 12 & 16 \\ 12 & 16 & 9 \\ 16 & 9 & 12 \end{matrix}\right] \\ \hline 
\end{array}$$

In \cite[Section~6]{Leh3}, Lehmer provided a converse to Chowla's result using Gauss's formulas for the cyclotomic numbers $(0, 0)$ and $(1, 0)$, showing that $|K|$ must be an odd perfect square if $K$ is a difference set. Here, we offer an alternative proof. Suppose that $K$ is a $(p, k, \lambda)$-difference set. According to Lemma~\ref{lem: basic properties for (i,j)} and the assumption $(0,0) = (1,0) = \lambda$, one obtains $$A = \left[\begin{matrix} \lambda & a & \lambda & b \\ \lambda & \lambda & b & a \\ \lambda & \lambda & \lambda & \lambda \\ \lambda & b & a & \lambda \end{matrix}\right] \quad \text{and} \quad B = \left[\begin{matrix} a & \lambda & b \\ \lambda & b & a \\ b & a & \lambda \end{matrix}\right],$$ where $2 \lambda + a + b = k = 4 \lambda + 1$. By Corollary~\ref{cor: A^T A for a difference set K}, we have $\lambda^2 + a^2 + b^2 = (k - \lambda) (\lambda + 1)$. Then $2 a b = (a+b)^2 - (a^2 + b^2) = 2 \lambda^2$. Thus, $a$ and $b$ are two integer solutions  of the quadratic equation $x^2 - (2 \lambda + 1) x + \lambda^2 = 0$, whose discriminant is $(2 \lambda + 1)^2 - 4 \lambda^2 = k$, showing that $k$ must be a perfect square. This recovers Lehmer's conclusion via matrix analysis. 


(II) 
Suppose $\ell = 8$. By \cite[Theorem~VII]{Leh3}, $K$ is a difference set if and only if $p = 1 + 8 t^2 = 9 + 64 s^2$ for some odd integers $t, s$. The smallest such prime is $p=73$, and with primitive root $5$, the corresponding matrices $A$ and $B$ are $$A = \left[\begin{matrix} 1 & 2 & 0 & 0 & 2 & 2 & 2 & 0 \\ 1 & 1 & 0 & 1 & 2 & 0 & 1 & 3 \\ 1 & 1 & 1 & 3 & 2 & 1 & 0 & 0 \\ 1 & 1 & 1 & 1 & 0 & 3 & 0 & 2 \\ 1 & 1 & 1 & 1 & 1 & 1 & 1 & 1 \\ 1 & 0 & 3 & 0 & 2 & 1 & 1 & 1 \\ 1 & 3 & 2 & 1 & 0 & 0 & 1 & 1 \\ 1 & 0 & 1 & 2 & 0 & 1 & 3 & 1 \end{matrix} \right] \quad \text{and} \quad B = \left[\begin{matrix} 2 & 0 & 0 & 2 & 2 & 2 & 0 \\ 1 & 0 & 1 & 2 & 0 & 1 & 3 \\ 1 & 1 & 3 & 2 & 1 & 0 & 0 \\ 1 & 1 & 1 & 0 & 3 & 0 & 2 \\ 0 & 3 & 0 & 2 & 1 & 1 & 1 \\ 3 & 2 & 1 & 0 & 0 & 1 & 1 \\ 0 & 1 & 2 & 0 & 1 & 3 & 1 \end{matrix} \right].$$ 
The next such prime is $p=104411704393$, which is too large to include the corresponding matrix $B$ here.
\end{exmp}

\begin{remk}
\label{remk: B^T B for a difference set K}
Corollary~\ref{cor: A^T A for a difference set K} plays a central role in the subsequent development, offering structural insights into $K$ when it is a difference set. This naturally raises the question of whether the converse holds when $|K|$ is odd---analogous to Theorem~\ref{thm: A^T A for a difference set K}---or whether an equivalent characterization can be formulated.
\end{remk}





Recall a classical result below.

\begin{thm}[{\cite[Theorem~1]{Schoenberg}}]
\label{thm: Schoenberg}
A regular $n$-simplex can be inscribed in the $\rr^n$ with integer coordinates if and only if $n$ is one of the following cases:
\begin{enumerate}
\item
$n$ is even and $n+1$ is a perfect square;
\item
$n \equiv 3 \pmod{4}$;
\item
$n \equiv 1 \pmod{4}$ and $n+1$ is a sum of two odd squares.
\end{enumerate}
\end{thm} 

The sequence of integer $n$ in the preceding theorem corresponds to \href{https://oeis.org/A096315}{A096315}~\cite{OEIS}.

Note that there is a regular $n$-simplex inscribed in $\zz^n$ if and only if there is a $n$-simplex inscribed in the $\qq^n$, since one may multiply all coordinates by a common denominator to obtain integer entries. Let $u_0, u_1, \ldots, u_n$ be column vectors of $\rr^n$ and let $V$ be the $n \times n$ matrix with columns $u_1 - u_0, \ldots, u_n - u_0$. Then $u_0, u_1, \ldots, u_n$ form the vertices of a regular $n$-simplex if and only if the Gram matrix $V^T V$ satisfies \begin{align} V^T V = r^2 \left( \begin{matrix} 1 & \frac{1}{2} & \cdots & \frac{1}{2} \\ \frac{1}{2} & 1 & \cdots & \frac{1}{2} \\ \vdots & \vdots & \ddots & \vdots \\ \frac{1}{2} & \frac{1}{2} & \cdots & 1 \end{matrix} \right) = \frac{r^2}{2} (J_n + I_n) \label{eq: Gram matrix from regular simplex} \end{align} where $r = ||u_i - u_0||$ is the edge length. Moreover, if all $u_i \in \qq^n$, then $\frac{r^2}{2} \in \qq$. 


\begin{thm}
\label{thm: if ell(k-lambda) is a square}
Suppose that $K$ is a $(q, k, \lambda)$-difference set. If $\ell (k - \lambda)$ {\rm(}$=q - k${\rm)} is a perfect square, then either $\ell \equiv 0 \pmod{4}$, or $\ell = a^2 + b^2$ for some odd integers $a$ and $b$.
\end{thm}

\begin{proof}
Let $n = \ell-1$ and let $N = k - \lambda$. From Corollary~\ref{cor: A^T A for a difference set K}, we have$B^T B = N (\lambda J_n + I_n)$. Consider the matrix expression \begin{align*} (B^T - z J_n) (B - z J_n) = (n z^2 - 2 N z + N \lambda) J_n + N I_n. \end{align*} Set $z = \frac{N}{n} + \frac{\sqrt{(n+1) N}}{n}$. Under the assumption that $(n+1) N = \ell (k - \lambda)$ is a perfect square, this choice of $z$ lies in $\qq$. This yields $(B^T - z J_n) (B - z J_n) = N (J_n + I_n)$, and it is exactly the Gram matrix of a regular $n$-simplex inscribed in $\qq^2$, as in~\eqref{eq: Gram matrix from regular simplex} with $r = \sqrt{2 N}$. As $n = \ell - 1$ is odd, it follows from Theorem~\ref{thm: Schoenberg}(2) and (3) that either $\ell \equiv 0 \pmod{4}$, or $\ell \equiv 2 \pmod{4}$ and $\ell$ is a sum of two odd squares of integers. 
\end{proof}

The prime powers $q$ of the form $q = 8 a^2 - 1$ for some $a \in \zz$ correspond to the case $\ell=2$ in Theorem~\ref{thm: if ell(k-lambda) is a square}. Examples include the primes $p = 7, 31, 71, 127, 199, 647, \ldots.$, which form the sequence \href{https://oeis.org/A090684}{A090684} in~\cite{OEIS}. The first non-prime example is $q = 12167$, arising from $a = 39$. The prime number $73$ with $\ell = 8$ also satisfies the condition of the theorem. However, it is currently unknown whether any example exists with $\ell = 4$.


Now considet the case where $\lambda$ is a finite geometric sum of even length: $\lambda = 1 + \ell + \cdots + \ell^e$ for some even integer $e \geq 0$. Then $k - \lambda = (\ell-1) \lambda + 1 = \ell^{e+1}$. So $\ell (k - \lambda) = \ell^{e+2}$, which is clearly a perfect square. This leads to the following proposition. 

\begin{prop}
\label{prop: lambda is a geometric series of ell}
Suppose that $K$ is a $(q, k, \lambda)$-difference set. If $\lambda = 1 + \ell + \cdots + \ell^e$ for some even $e \geq 0$, then either $\ell \equiv 0 \pmod{8}$, or $\ell = a^2 + b^2$ for some odd integers $a$ and $b$.
\end{prop}

\begin{proof}
Since such $\lambda$ is always odd, Lemma~\ref{lem: difference set with odd lambda} implies that $\ell \equiv 0, 2 \pmod{8}$. The conclusion follows by applying Theorem~\ref{thm: if ell(k-lambda) is a square}.
\end{proof}

As an immediate corollary:

\begin{cor}
\label{cor: finite projective plane}
If $K$ is a $(q, k, 1)$-difference set, then either $8 \mid \ell$, or $\ell$ is a sum of two odd squares.
\end{cor}


Examples satisfying the proposition include the cases: $\linebreak (q, k, \lambda) = (31, 15, 2), (127, 63, 31), (8191, 4095, 2047), (73, 9, 1)$.


\subsection{Two additional properties}
\label{subsec: Two additional properties}

\mbox{}

From Corollary~\ref{cor: A^T A for a difference set K}, one obtains the identity \begin{align*} (B^T - \lambda J_{\ell-1}) (B - \lambda J_{\ell-1}) = (k - \lambda) I_{\ell-1} - \lambda J_{\ell-1}. 
\end{align*} Focus on the $(q', q')$-entry of the left-hand side, which equals $\sum_{j=0}^{\ell-1} ((j, q')-\lambda)^2 - ((q', q')-\lambda)^2$. Since $(j, q') \overset{*}{=} (0, j + q')$ and $(q', q') \overset{*}{=} (0, 0) = \lambda$, this simplifies to \begin{align}\sum_{j=1}^{\ell-1} ((0, j) - \lambda)^2 = k - 2 \lambda. \label{eq: sum of squares of (0,j)-lamnda} \end{align} This identity concludes the following result that provides a missing piece about $\lambda$ in Lemma~\ref{lem: difference set with odd lambda}(i). 

\begin{prop}
\label{prop: odd lambda}
If $K$ is a $(q,k,\lambda)$-difference set with odd $\lambda$ and $\ell \equiv 0 \pmod{8}$, then $\lambda \equiv 1 \pmod{4}$.
\end{prop}

\begin{proof}
Since $(0,0) = \lambda$ is odd, every $(0,j)$ is even for $j = 1, \ldots, \ell-1$ by Lemma~\ref{lem: Lemma I in Leh3} (Lehmer). Thus, $\sum_{j=1}^{\ell-1} ((0,j) - \lambda)^2 \equiv \ell-1 \pmod{8}$. It follows from~\eqref{eq: sum of squares of (0,j)-lamnda} that $k - 2 \lambda \equiv \ell - 1 \pmod{8}$. The result follows since $k = \lambda \ell + 1$. 
\end{proof}

A second property arises when $\lambda = 1$, the case corresponding to finite projective planes. 

\begin{prop}
\label{prop: lambda = 1}
If $K$ is a $(q, k, 1)$-difference set, then $(0, j) = 0$ or $2$ for all $j = 1, \ldots, \ell-1$.
\end{prop}

\begin{proof}
If $\lambda = 1$, then $\sum_{j=1}^{\ell-1} ((0,j) - 1)^2 = k - 2 \lambda = \ell - 1$ by~\eqref{eq: sum of squares of (0,j)-lamnda}. Since each term $((0,j) - 1)^2$ is a positive odd integer, it follows that $((0,j) - 1)^2 = 1$ for each $j$, as desired. 
\end{proof}

An explicit example appears in Example~\ref{exmp: B^TB to ell=4 and 8}(II), where $\ell = 8$ and $\lambda = 1$. It would be interesting to explore whether Proposition~\ref{prop: lambda = 1} admits a combinatorial explanation in $\ff_q$ as a finite projective plane raised from $K$.

\subsection{Shifted cyclotomic matrices}
\label{subsec: Shifted cyclotomic matrices}

\mbox{}

In this subsection, we consider a variant of the cyclotomic matrix obtained by a specific row permutation, which yields a symmetric matrix whose spectral structure can be explicitly described when $K$ is a difference set. We hope this construction may find applications in the future study of power difference sets. 

Throughout this subsection, we assume that $K$ is a $(q,k,\lambda)$-difference set in $\ff_q$. Under this assumption, one has that $\ell$ is even and $q' \equiv \frac{\ell}{2} \pmod{\ell}$. Moreover, $\ell \mid k-1$ and $\lambda = (i, 0) = \frac{k-1}{\ell}$ for every $i$. 

For the cyclotomic matrix $A$, we consider the matrix $M = P A$, where $P = P_{q'}$ is the permutation matrix described in Remark~\ref{remk: conjugates of A_0}. Each $(i, j)$-entry of $M$ is the cyclotomic number $(i + q', j)$. In other words, the row of $M$ indexed by $i$ is exactly the row of $A$ indexed by $i+q'$ modulo $\ell$. Hence, $M$ is obtained from $A$ by exchanging the top $q'$ rows (indices $0, \ldots, q'-1$) with the bottom $\ell-q'$ rows (indices $q', \ldots, \ell-1$). This matrix $M$ is exactly the matrix $[c_{s,t}]$ introduced by Snapper in Remark~\ref{remk: Snapper}. Furthermore, the identity $(i + q', j) \overset{*}{=} (j + q', i)$ implies that $M$ is symmetric. Since $P^T = P = P^{-1}$, it follows that $M^2 = M^T M = (P A)^T (P A) = A^T A$. Therefore, Theorem~\ref{thm: A^T A for a difference set K} provides \begin{align} M^2 = \lambda k J_{\ell} + (k-\lambda) I_{\ell} - k E_{0,0}. \label{eq: M^2}\end{align}

On the other hand, the matrix $M$ has the block form \begin{align} M = \left( \begin{array}{c|ccc} \lambda & \lambda \mb{1}_{n}^T \\ \hline \lambda \mb{1}_{n} & S \end{array} \right) \label{eq: relation between M and S} \end{align} where $S$ is the $(0,0)$-minor of $M$, $n = \ell-1$ and $\mb{1}_{n}$ is the $n \times 1$ all-one column. Thus, \begin{align*} M^2 = \left( \begin{array}{c|ccc} 
\lambda & \lambda \mb{1}_n^T \\ \hline 
\lambda \mb{1}_{n} & S 
\end{array} \right) \cdot
\left( \begin{array}{c|ccc} 
\lambda & \lambda \mb{1}_n^T \\ \hline 
\lambda \mb{1}_{n} & S 
\end{array} \right)
=
\left( \begin{array}{c|ccc} 
\lambda^2 \ell & \lambda k \mb{1}_n^T \\ \hline 
\lambda k \mb{1}_{n} & \lambda^2 J_n + S^2 
\end{array} \right),
\end{align*} 
since $S \mb{1}_n = \lambda (k-\lambda) \mb{1}_n$.  Combining \eqref{eq: M^2}, we obtain \begin{align} S^2 = \lambda k J_n + (k-\lambda) I_n - \lambda^2 J_n = (k - \lambda) (\lambda J_n + I_n). \label{eq: square of S} \end{align} In fact, $S$ is of the form $P' B$ for some $P'$ satisfying $P'^T = P'^{-1}$, analogous to the relation $M = PA$. Consequently, $S^2 = B^T B$, which shares the same structure described in Corollary~\ref{cor: A^T A for a difference set K}. 

For a matrix $T$, we denote its spectrum by $\Spec(T) = \{a_1^{(n_1)}, a_2^{(n_2)}, \ldots\}$, where the $a_i$ are distinct eigenvalues and $n_i$ denotes the multiplicities of $a_i$. If $n_i = 1$, we simply write $a_i^{(1)} = a_i$ . We will first determine the spectrum of $S$, and then use it to obtain the spectrum of $M$.



\begin{lem}
\label{lem: spectrum of S}
The spectrum of $S$ is given by $$\Spec(S) = \left\{k - \lambda, \sqrt{k - \lambda}^{\left(\frac{\ell}{2}-1\right)}, -\sqrt{k - \lambda}^{\left(\frac{\ell}{2}-1\right)} \right\},$$ where each eigenvector corresponding to $\pm \sqrt{k - \lambda}$ has entries summing to zero. Moreover, the minimal polynomial of $S$ is $(x - (k - \lambda)) (x^2 - (k - \lambda))$.
\end{lem}

\begin{proof}
Let $n = \ell-1$, which is odd, and denote $N = k - \lambda$. From~\eqref{eq: square of S}, one has $\Spec(S^2) = \{{N^2}, N^{(n-1)}\}$ where, the eigenvalue $N^2$ corresponds the one-dimensional eigenspace spanned by ${\bf 1}_n$, and $N$ corresponds an $(n-1)$-dimensional eigenspace consisting of vectors whose entries sum to zero. Since each row sum of $S$ is $N$, it follows that $\Spec(S) = \{N, \sqrt{N}^{(a)}, - \sqrt{N}^{(b)}\}$ for non-negative integers $a, b$ with $a + b = n-1$. To determine $a$ and $b$, consider the trace of $S$: $$\sum_{i=1}^{\ell-1} (i + q', i) \overset{**}{=} \sum_{i=1}^{\ell-1} (-i + q', q') = k - (q', q') = N,$$ where the second equality follows from Lemma~\ref{lem: basic properties for (i,j)}(iv). 
This implies that $a = b = \frac{n-1}{2} = \frac{\ell}{2} - 1$. Moreover, if $w$ is an eigenvector of $S$ corresponding to the eigenvalue $\eta = \pm \sqrt{N}$, then ${\bf 1}_n^T (S w) =  \eta {\bf 1}_n^T w$. On the other hand, $({\bf 1}_n^T S) w = N {\bf 1}_n^T w$. Since $N \neq \eta$, it follows that ${\bf 1}_n^T w = 0$. 

Finally, from~\eqref{eq: square of S}, we have $(S - N I_n) (S^2 - N I_n) = 0$ because $S J_n = N J_n$, which derives the minimal polynomial of $S$.
\end{proof}

\begin{prop}
\label{prop: spectrum of M}
The spectrum of $M$ is given by $$\Spec(M) = \left\{\frac{k + \sqrt{k^2 - 4 \lambda}}{2}, \frac{k - \sqrt{k^2 - 4 \lambda}}{2}, \sqrt{k - \lambda}^{\left(\frac{\ell}{2}-1\right)}, -\sqrt{k - \lambda}^{\left(\frac{\ell}{2}-1\right)} \right\},$$ where each eigenvector corresponding to $\pm \sqrt{k - \lambda}$ is of the form $(0, w)$ for some vector $w \in \rr^{\ell-1}$ whose entries sum to zero. Moreover, the minimal polynomial of $M$ is $(x^2 - k x + \lambda)) (x^2 - (k - \lambda))$.
\end{prop}

\begin{proof}
Let $w$ be an eigenvector of $S$ with respect to the eigenvalue $\eta = \pm \sqrt{k - \lambda}$. By \eqref{eq: relation between M and S} and Lemma~\ref{lem: spectrum of S}, one has \begin{align*}M \cdot \left( \begin{array}{c} 0 \\ w \end{array} \right) = \left( \begin{array}{c} \lambda {\bf 1}_n^T w \\ S w \end{array} \right) = \eta \left( \begin{array}{c} 0 \\ w \end{array} \right). \end{align*} Therefore, $\eta$ is also an eigenvalue of $M$, and its multiplicity is at least $\frac{\ell}{2}-1$. To determine the complete spectrum of $M$, it suffices to show that the minimal polynomial of $M$ is as stated. To this end, observe that $$J_{\ell} E_{0,0} = \sum_{i=0}^{\ell-1} E_{i, 0}, \quad (E_{0,0} + M) J_{\ell} = k J_{\ell}, \quad (E_{0,0} + M) E_{0,0} = E_{0,0} + \lambda \sum_{i=0}^{\ell-1} E_{i, 0}.$$ using the identity from~\eqref{eq: M^2}, we obtain \begin{align*} \left( M^2 - k M + \lambda I_{\ell} \right) \left( M^2 - (k - \lambda) I_{\ell}\right) = \left(\lambda k J_{\ell} + k I_{\ell} + k (E_{0,0} +M) \right) \left( \lambda k J_{\ell} - k E_{0,0} \right) = 0, \end{align*} where we have used the relation $\lambda \ell + 1 = k$. Since the polynomial $x^2 - k x + \lambda$ has two distinct roots $\frac{k \pm \sqrt{k^2 - 4 \lambda}}{2}$ (as $k$ is odd), the result follows.
\end{proof}

\begin{exmp}
\label{exmp: spectra of M and S to q=73}
According to Example~\ref{exmp: B^TB to ell=4 and 8}(II), the matrices $M$ and $S$ for $q = 73$ are given by $$M = \left[\begin{matrix} 1 & 1 & 1 & 1 & 1 & 1 & 1 & 1 \\ 1 & 0 & 3 & 0 & 2 & 1 & 1 & 1 \\ 1 & 3 & 2 & 1 & 0 & 0 & 1 & 1 \\ 1 & 0 & 1 & 2 & 0 & 1 & 3 & 1 \\ 1 & 2 & 0 & 0 & 2 & 2 & 2 & 0 \\ 1 & 1 & 0 & 1 & 2 & 0 & 1 & 3 \\ 1 & 1 & 1 & 3 & 2 & 1 & 0 & 0 \\ 1 & 1 & 1 & 1 & 0 & 3 & 0 & 2 \end{matrix} \right] \quad \text{and} \quad S = \left[\begin{matrix} 0 & 3 & 0 & 2 & 1 & 1 & 1 \\ 3 & 2 & 1 & 0 & 0 & 1 & 1 \\ 0 & 1 & 2 & 0 & 1 & 3 & 1 \\ 2 & 0 & 0 & 2 & 2 & 2 & 0 \\ 1 & 0 & 1 & 2 & 0 & 1 & 3 \\ 1 & 1 & 3 & 2 & 1 & 0 & 0 \\ 1 & 1 & 1 & 0 & 3 & 0 & 2 \end{matrix} \right],$$ where $$\Spec(M) = \left\{\frac{9+\sqrt{77}}{2}, \frac{9-\sqrt{77}}{2}, 2\sqrt{2}^{(3)}, -2\sqrt{2}^{(3)}\right\}$$ and $$\Spec(S) = \left\{8, 2\sqrt{2}^{(3)}, -2\sqrt{2}^{(3)}\right\}.$$
\end{exmp}

We are now able to determine the determinants of the cyclotomic matrix $A$ and its submatrix $B$.

\begin{cor}
\label{cor: det of A}
$\det(A) = - \lambda (k - \lambda)^{\frac{\ell}{2}-1}$ and $\det(B) = (-1)^{\frac{\ell}{2} - 1} (k - \lambda)^{\frac{\ell}{2}}$.
\end{cor}

\begin{proof}
Recall that $M$ is obtained from $A$ by exchanging the rows indexed by $0, \ldots, q' - 1$ with those indexed by $q', \ldots, \ell-1$. Thus, $\det(M) = (-1)^{q'^2} \det(A) = (-1)^{q'} \det(A)$. Form Proposition~\ref{prop: spectrum of M}, we know that $\det(M) = \lambda (-1)^{\frac{\ell}{2}-1} (k - \lambda)^{\frac{\ell}{2}-1}$, and since $q' = \frac{\ell}{2}$, we obtain the expression for $\det(A)$. Similarly, $S$ is obtained from $B$ by exchanging the rows indexed by $0, \ldots, q' - 1$ with those indexed by $q'+1, \ldots, \ell-1$. Hence, $\det(S) = (-1)^{q' (q'-1)} \det(B) = \det(B)$, and the result follows from Lemma~\ref{lem: spectrum of S}. 
\end{proof}

In general, the spectrum of the cyclotomic matrix $A$ tends to be complicated. For instance, when $q$ is a prime and $\ell = 4$, one can deduce from~Example~\ref{exmp: B^TB to ell=4 and 8}(I) that the characteristic polynomial $A$ is $x^4 - (k-1) x^3 - \frac{k-1}{2} x^2 - \lambda (k+1) x - \lambda (k - \lambda)$. As a concrete example, for $q = 37$, the characteristic polynomial of $A$ is $x^4 - 8 x^3 - 4 x^2 - 20 x - 14$ which is irreducible over rational numbers.




\subsection{Modified power difference set}
\label{subsec: Modified power difference set}

\mbox{}

A power difference set consists only of nonzero elements, whereas a {\it modified power difference set} includes the zero element as well. Denote $K_0 = K \cup \{0\}$. Lehmer also determined the conditions under which $K_0$ is a difference set in terms of cyclotomic numbers. We restate and prove her result below, as the original formulation was not presented in full detail. 

\begin{thm}[{\cite[Theorem~III$'$]{Leh3}}]
\label{thm: Theorem III' in Leh3}
Let $K_0$ consist of all $\ell$-th powers of $\ff_q$, and let $k_0 = |K_0|$. Then $K_0$ is a difference set of $\ff_q$ if and only if $k_0$ is even, $\ell$ is even and $(0,0) + 1 = (\frac{\ell}{2}, 0) + 1 = (i, 0)$ for $1 \leq i \neq \frac{\ell}{2} \leq \ell-1$. In this case, $(i, 0) = \frac{k_0}{\ell}$.
\end{thm}

\begin{proof}
Note that $k = |K| = k_0 - 1$. Let $z \in g^i K$ for $0 \leq i \leq \ell-1$. Since $K_0 = K \cup \{0\}$, we consider the possibilities where $z = x - y$ for $(x, y) \in K \times K$, $(x, y) = (0, y)$, or $(x, y) = (x, 0)$. Following the same reasoning as in the proof of Theorem~\ref{thm: Theorem I, II, III in Leh3}, one has $$\left| \left\{ (x, y) \in K_0 \times K_0 \mid x - y = z \right\} \right| = \left\{ \begin{array}{ll} (i, 0) + 2 & \text{if $g^i K = K = -K$,} \\ (i, 0) + 1 & \text{if $g^i K = K$ or $-K$, and $K \neq - K$,} \\ (i, 0) & \text{otherwise.} \end{array} \right.$$ Note that the cases $g^i K = K$ and $g^i = -K$ correspond to $i = 0$ and $i = q'$, respectively. Suppose now that $K_0$ is a difference set. If $k$ even, then $K = -K$, and we have $(0, 0) + 2 = (1, 0) = \cdots = (\ell-1, 0)$. In this situation, $q' = 0$ and $(i, 0) \overset{*}{=} (0, i)$. So all $(0, i)$ have the same parity, which contradicts to Lemma~\ref{lem: Lemma I in Leh3}. Therefore, $k$ must be odd, $\ell = \frac{q-1}{k}$ is even, $K \neq -K$ and $q' = \frac{\ell}{2}$. Moreover, $(0, 0)+1 = (q', 0) + 1 = (i, 0)$ for all $1 \leq i \neq \frac{\ell}{2} \leq \ell-1$, and these equalities imply conversely that $K_0$ is a difference set. Let $\lambda_0 = (i, 0)$. Then $k - 1 = \sum_{i=0}^{\ell-1} (i, 0) = \ell \lambda_0 - 2$, which gives $\lambda_0 = \frac{k+1}{\ell}$.
\end{proof}

When $\ell=2$, one can deduce from Lemma~\ref{lem: basic properties for (i,j)} that $k$ is odd if and only if $(0, 0) = (1, 0)$. Thus, $K_0$ is a difference set if and only if $q \equiv 3 \pmod{4}$. 
A detailed discussion of when $K_0$ is a difference set for $\ell = 4$ or $8$ can be found in \cite[Section~7]{Leh3}.

Following a similar approach as in the proof of Theorem~\ref{thm: A^T A for a difference set K}, we can derive an expression for $A^T A$ when $A$ is the cyclotomic matrix associated with a modified power difference set.

\begin{prop}
\label{prop: A^T A for a modified difference set K}
Let $K_0$ be the set of all $\ell$-th powers in $\ff_q$ with $\ell \mid q-1$, and let $A$ be a cyclotomic matrix with respect to $\ell$. If $K_0$ is a $(q, k_0, \lambda_0)$-difference set in $\ff_q$, then $$A^T A = \lambda_0 (k_0-1) J_{\ell} + (k_0 - \lambda_0-1) I_{\ell} - (k_0-1) E_{0,0} - A - A^T.$$
\end{prop}

The preceding equation can be rearranged into a more compact form: $$(A^T + I_{\ell}) (A + I_{\ell}) = \lambda_0 (k_0-1) J_{\ell} + (k_0 - \lambda_0) I_{\ell} - (k_0-1) E_{0,0}.$$ Moreover, let $B_0$ be the $(q', 0)$-minor of $A + I_{\ell}$. Then we further obtain $$B_0^T B_0 = \lambda_0 (k_0 - \lambda_0-1) J_{\ell-1} + (k_0 - \lambda_0) I_{\ell-1}.$$ We hope this expression will be useful in future investigations of modified power difference sets.



\section{Concluding Remark}
\label{sec: Concluding Remark}

We observed below Lemma~\ref{lem: cyclotomic numbers and Schur ring} that the basic operations of Schur rings yield certain properties of cyclotomic numbers (Lemma~\ref{lem: basic properties for (i,j)}). It is natural to ask whether cyclotomic numbers can be further understood through the algebraic structure of Schur rings. Such structures may impose additional constraints on cyclotomic numbers.  

The reader may notice that all examples of power difference sets with $\ell = 4$ or $\ell = 8$ presented in this article occur in prime fields. Lehmer's original proof on the sufficient and necessary conditions for the existence of $4$-th and $8$-th power difference sets in prime fields can be extended to general finite fields $\ff_q$. Specifically, the condition for $4$-th powers is $q = 1 + 4 t^2$ for odd $t$, and for $8$-th powers is $q = 1 + 8 t^2 = 9 + 64 s^2$ for odd $t$ and $s$. However, it has been shown by M. Hall \cite[Theorem~4.1]{MHall3} and T. Storer \cite[Theorem~20]{Sto2} that these cases occur only when $q$ is prime. This leads to the following interesting problem.

\begin{prob}
If $K$ is a power difference set of $\ff_q$ with $\ell > 2$, must $q$ be a prime?
\end{prob}


From Lehmer's conditions above, one observes that $|K| = t^2$ is a perfect square when $K$ is a difference set for $\ell=4$ or $\ell=8$. 
This naturally leads to the question of whether this phenomenon holds in general.

\begin{prob}
If $K$ is a power difference set of $\ff_q$ with $\ell > 2$, must $|K|$ be a perfect square?
\end{prob}

If the answer is affirmative, then $k \equiv 1 \pmod{8}$, implying $\lambda \ell \equiv 0 \pmod{8}$. In particular, this would yield $\ell \equiv 0 \pmod{8}$ when $\lambda$ is odd (e.g. $\lambda = 1$ in the case of a projective plane). 

\section*{Acknowledgment}

The author acknowledges support from National Science and Technology Council (NSTC) grants 112-2811-M-003-017 and 114-2115-M-017-001-MY2.



\bibliographystyle{alphaurl} 
\bibliography{../../../research/bib/Bib} 

\end{document}